\documentclass[12pt]{amsart}
\usepackage{amsmath,txfonts}
\usepackage{amssymb}
\usepackage{amsxtra}
\usepackage{amsthm, color}
\usepackage{txfonts}
\usepackage{graphicx}
\usepackage{times}
\usepackage{citeref}
\usepackage{tikz}
\usepackage{hyperref}
\usepackage{stmaryrd}
\usepackage[T3,T1]{fontenc}
\DeclareSymbolFont{tipa}{T3}{cmr}{m}{n}
\DeclareMathAccent{\invbreve}{\mathalpha}{tipa}{16}

\usepackage{pgfplots}

\allowdisplaybreaks

\makeatletter
\@namedef{subjclassname@1991}{Subject}
\@namedef{subjclassname@2000}{Subject}
\@namedef{subjclassname@2010}{Subject}
\@namedef{subjclassname@2020}{{\rm 2020} Mathematics Subject Classification}
\makeatother

\def\rr{{\mathbb R}}
\def\rn{{{\rr}^n}}

\def\XXint#1#2#3{{\setbox0=\text{$#1{#2#3}{\int}$ }
		\vcenter{\text{$#2#3$ }}\kern-.6\wd0}}

\newtheorem{theorem}{Theorem}[section]
\newtheorem{definition}[theorem]{Definition}

\newtheorem{corollary}[theorem]{Corollary}
\newtheorem{proposition}[theorem]{Proposition}
\newtheorem{lemma}[theorem]{Lemma}

\theoremstyle{definition}

\newcommand{\lt}[1]{[ {#1}] \lower.3ex\text{$_{t}$}}

\begin{document}
	\title[A positive solution to the $L^p$ projection centroid conjecture]
	{A positive solution to the $L^p$ projection centroid conjecture}

	\author{Jin Dai}
	
        \address{School of Mathematical Sciences, Chongqing University of Technology, Chongqing, 400054, China}
	            \email{daijin@cqut.edu.cn}
	
	\author{Tuo Wang}
	
        \address{School of Mathematics and Statistics, Shaanxi Normal University, Xi'an, 710119, China}
	            \email{wang.tuo@snnu.edu.cn}
	
	

	\thanks{JD was supported by the Scientific Research Foundation of Chongqing University of Technology \# 2023ZDZ037
; TW was supported by National Natural Science Foundation of China (Grant No.12471055), the Recruitment Program for Young Professionals of China, Youth Innovation Team of Shaanxi Universities, Shaanxi Fundamental Science Research Project for Mathematics and Physics (Grant No. 22JSZ012) and the Fundamental Research Funds for the Central Universities (Grant No. GK202603001, GK202307001, GK202202007)
.}

	\date{\today}



	\begin{abstract} In a classical paper \cite{LYZ00} in 2000, Lutwak-Yang-Zhang established the $L^p$ analog of the Petty projection inequality and the $L^p$ analog of the Busemann-Petty centroid inequality. In Section 7 of \cite{LYZ00}, Lutwak-Yang-Zhang proposed the important $L^p$ projection centroid conjecture. We give a positive solution to the $L^p$ projection centroid conjecture in this work.
				\end{abstract}

	\maketitle



\tableofcontents
	
	\arraycolsep=1pt
	\numberwithin{equation}{section}

\section{Introduction} \label{s1}
We will need some definitions before we introduce the $L^p$ projection centroid conjecture proposed by Lutwak-Yang-Zhang in 2000.

We will work in Euclidean space $\rn,n\geq 2$ throughout the paper. The standard inner product of two vectors $x,y\in\mathbb R^n$ is denoted by $x\cdot y$. For $x\in\rn$, we use $|x|=\sqrt{x\cdot x}$ to denote the Euclidean norm of $x$.

A convex body in $\rn$ is a compact convex set with nonempty interior. Denote by $\mathcal{K}^n$, $\mathcal{K}_{o}^n$ and $\mathcal{K}_{e}^n$  the set of all convex bodies, the set of all convex bodies with the origin $o$ contained in their interiors  and the set of all origin-symmetric convex bodies  in $\mathbb R^n$, respectively. The unit ball is denoted by $B^n$ and write $S^{n-1}$ for the unit sphere in $\rn$.

The convex body $K\in\mathcal{K}^n$ is uniquely determined by its support function, defined by
$$h_K(x)=\sup\big\{x\cdot y:y\in K\big\},~\forall x\in\rn.$$

The radial function, $$\rho_K(x)=\sup\big\{\lambda\geq 0:\lambda x\in K\big\},~\forall x\in\rn\backslash\{o\},$$ of a compact, star-shaped (with respect to the origin) $K\subseteq\mathbb R^n$, is defined. If $\rho_K$ is positive and continuous, call $K$ a star body (with respect to the origin).

Denote the class of star bodies (with respect to the origin) in $\mathbb R^n$ by $\mathcal S_{o}^n$. It is easy to see that the star body $K\in\mathcal{S}_{o}^n$ is uniquely determined by its radial function $\rho_K$. The star body $K\in\mathcal{S}_{o}^n$ is also uniquely determined by its gauge function defined by
$$|x|_K=\inf\big\{\lambda\geq 0:x\in \lambda K\big\},~\forall x\in\rn.$$

When $K\in\mathcal K_{o}^n$, define the polar body $K^\circ$ of $K$ by
$$K^\circ=\big\{x\in\rn:x\cdot y\leq 1\ \mathrm{for}\ \mathrm{all}\ y\in K\big\}.$$

Clearly, we have $$h_K(x)=|x|_{K^\circ}=\frac{1}{\rho_{K^\circ}(x)},$$ for $x\in\rn\backslash\{o\}$ and $K\in\mathcal K^n_o$.

\subsection{The projection centroid conjecture}
In convex geometry, the centroid body operator is a classical operator. Centroid bodies were attributed by Blaschke to Dupin. A famous affine isoperimetric inequality for centroid bodies is the Busemann-Petty centroid inequality. The projection body operator was introduced by Minkowski at the turn of the last century. A famous affine isoperimetric inequality for projection bodies is the Petty projection inequality (see \cite{Petty85,Sch14}). The centroid body operator was shown to be strongly connected to the projection body operator in 1993 by Lutwak \cite{lutwak93}. Lutwak in \cite{lutwak93} also made the projection centroid conjecture.

Let $K\in\mathcal{K}^n$. The projection body ${\Pi}K$ of $K$ is defined by its support function
$$h_{\Pi K}(x)=\frac{1}{2} \int_{S^{n-1}} |\xi\cdot x|dS_{K}(\xi),\quad x\in\rn,$$
where $S_K$ is the Alexandrov-Fenchel-Jessen surface area measure of $K$, see (\ref{96}).

The centroid body ${\Gamma}K$ of $K\in\mathcal S_{o}^n$ is defined by its support function
$$h_{\Gamma K}(x)=\frac{1}{(n+1)|K|} \int_{S^{n-1}} \rho^{n+1}_K(\xi)|\xi\cdot x|d\xi,\quad x\in\rn,$$
where $|K|$ denotes the volume of $K$.

For simplicity, we write ${\Pi}^\circ K$ and ${\Gamma}^\circ K$ for $({\Pi}K)^\circ$ and $({\Gamma}K)^\circ$, respectively.

In 1993, the projection centroid conjectures of Lutwak stated in \cite{lutwak93} are as follows:
(1) If the convex body $K\in\mathcal{K}_{o}^n$ is such that $K$ and ${\Gamma}{\Pi}^\circ K$ are dilates, must $K$ be an ellipsoid?
(2) If the star body $K$ is such that $K$ and ${\Pi}^\circ{\Gamma} K$ are dilates, must $K$ be an ellipsoid?

For the 2-dimensional case, the answer to both questions is much easier and is positive as shown by Ivaki \cite{Ivaki17}. For the general $n$-dimensional case, in 2017 a related result is got by Ivaki \cite{Ivaki17} in which the projection centroid conjectures are proved locally. Recently, both conjectures are completely resolved by Milman-Shabelman-Yehudayoff \cite{MSY}, and both answers are positive.

\subsection{The $L^p$ projection centroid conjecture}

As observed by Schneider \cite{Sch14}, the Brunn-Minkowski theory springs from joining the notion of ordinary volume in the $n$-dimensional Euclidean space, $\mathbb R^n$, with that of \emph{Minkowski combinations} of convex bodies. In the middle of the last century, Firey \cite{firey-1962}  first defined and studied the $p$-means of convex bodies and Lutwak \cite{LutwakL_p1,LutwakL_p2} established the Brunn-Minkowski-Firey theory by using the $p$-means to study the $L^p$ mixed volume, the $L^p$ Minkowski problem, the $L^p$ affine surface area, the $L^p$ geominimal surface area. During the past three decades various elements of the $L^p$ Brunn-Minkowski theory have attracted increased attention (see, e.g., \cite{2,3,4,7,9,KO,lin17,11,12,14,16,19,23,22,24,28,WangandXiao}).

The $L^p$ projection body operator and the $L^p$ centroid body operator were introduced by Lutwak-Yang-Zhang in 2000 in \cite{LYZ00}, where they established the $L^p$ analog of the Petty projection inequality and the $L^p$ analog of the Busemann-Petty centroid inequality. Now, the $L^p$ projection body operator and the $L^p$ centroid body operator have already become classical operators in the Brunn-Minkowski-Firey theory.

Let $K\in\mathcal{K}_{o}^n$ and $p>1$. The $L^p$ projection body ${\Pi}_pK$ of $K$ is defined by its support function
$$h_{\Pi_p K}^p(x)=\frac{1}{d_{n,p}} \int_{S^{n-1}} |\xi\cdot x|^pdS_{p,K}(\xi),\quad x\in\rn,$$
where $S_{p,K}$ is the $L^p$ surface area measure of $K$ and $d_{n,p}$ is the normalization constant, see (\ref{96}) and (\ref{95}).
The $L^p$ centroid body ${\Gamma}_pK$ of $K\in\mathcal S^n_o$ is defined by its support function
$$h_{\Gamma_p K}^p(x)=\frac{1}{b_{n,p}|K|} \int_{S^{n-1}} \rho^{n+p}_K(\xi)|\xi\cdot x|^pd\xi,\quad x\in\rn,$$
where $|K|$ denotes the volume of $K$ and $b_{n,p}$ is the normalization constant, see (\ref{94}).

For simplicity, we use ${\Pi}^\circ_pK$ and ${\Gamma}^\circ_pK$ to denote $({\Pi}_pK)^\circ$ and $({\Gamma}_pK)^\circ$, respectively.

For $p>1$, the $L^p$ projection centroid conjecture could be stated as follows:
(3) If the convex body $K\in\mathcal K_{o}^n$ is such that $K$ and ${\Gamma}_p{\Pi}^\circ_pK$ are dilates, must $K$ be an ellipsoid?
(4) If the star body $K\in\mathcal S_{o}^n$ is such that $K$ and ${\Pi}_p^\circ{\Gamma}_pK$ are dilates, must $K$ be an ellipsoid?

The question (4) is also resolved by Milman-Shabelman-Yehudayoff \cite{MSY}. The question (3) is the following conjecture raised by Lutwak-Yang-Zhang \cite{LYZ00} in 2000 and remains open:

\vspace{0.5em}

\noindent{\bf Conjecture }(Lutwak-Yang-Zhang).
{\it
Let $K\in\mathcal{K}_{o}^n$ and $p> 1$. If K is such that ${\Gamma}_p {\Pi}_p^\circ K$ is a dilate of $K$, then $K$ must be an ellipsoid.
}

\vspace{0.5em}

Using the idea of the continuous Steiner symmetrization of Milman-Shabelman-Yehudayoff \cite{MSY,MY23}, we solve this conjecture and give a positive answer to the conjecture in this paper:
\begin{theorem}\label{14}
Let $K\in\mathcal{K}_{o}^n$ and $p> 1$. If there is a constant $c>0$ such that ${\Gamma}_p {\Pi}_p^\circ K=cK$, then $K$ is an origin-symmetric ellipsoid.
\end{theorem}
\section{Preliminaries}

In this section, we mainly collect the basic materials from convex geometry. For more information, please see \cite{Gardner06,Gruber07,Sch14}.

The scale of a set $A\subseteq \rn$ is denoted by
$$\lambda A=\big\{\lambda a: a\in A\big\}$$
for real numbers $\lambda$, write $-A$ for $(-1)A$ and $\lambda A$ is also called the dilation if $\lambda>0$.
For $k\in\mathbb{N}$, the $k$-dimensional unit ball of $\mathbb{R}^k$ is denoted by $B^k$ and the $k$-dimensional volume $|B^k|$ of $B^k$ is denoted by $\omega_k$. For $p\geq 1$, let
\begin{equation}\label{94}
  b_{n,p}=\frac{(n+p)\omega_{n+p}}{\omega_2\omega_n\omega_{p-1}}
\end{equation}
and
\begin{equation}\label{95}
   d_{n,p}=\frac{2\omega_{n+p-2}}{\omega_{p-1}},
\end{equation}
where
$\omega_{q}=\pi^{q/2}/\Gamma(1+\frac{q}{2})$, $q\geq 0$ is a real number and $\Gamma(\cdot)$ is the standard gamma function. 

\subsection{Convex bodies}
We endow the space $\mathcal{K}^n$ with the natural topology induced by the Hausdorff metric, that is, a sequence $\{K_i\}$ of convex bodies converges to a convex body $K$, denoted by $K_i\rightarrow K$, with respect to the Hausdorff metric if and only if $h_{K_i}\rightarrow h_K$ uniformly on $S^{n-1}$ as $i\rightarrow\infty$.

Recall that for a Borel set $\omega\subseteq S^{n-1}$ where $ S^{n-1}$ denotes the unit sphere, $S_K(\omega)$ is the $(n-1)$-dimensional Hausdorff measure $\mathcal H^{n-1}$ of the set of all boundary points of $K$ for which there exists a normal vector of $K$ belonging to $\omega.$ $S_K(\cdot)$ is the surface area measure of $K$, also called the Alexandrov-Fenchel-Jessen surface area measure of $K$.

Let $\nu_K:\partial K\rightarrow S^{n-1}$ be the generalized Gauss map of $K\in\mathcal{K}^n$, we can define the $L^p$ surface area measure of convex bodies as follows:
given $p\geq 1$, associate $K\in\mathcal{K}^n_o$ with a Borel measure $S_{p,K}(\cdot)$ on $S^{n-1}$ called the $L^p$ surface area measure of $K$, defined by
\begin{equation}\label{96}
  S_{p,K}(\omega)=\int_{\nu^{-1}_K(\omega)}h^{1-p}_K(\nu_K(x))d\mathcal{H}^{n-1}(x),
\end{equation}
for each Borel set $\omega\subseteq S^{n-1}$, where $\mathcal{H}^{n-1}$ denotes the $(n-1)$-dimensional Hausdorff measure on the boundary of $K$.
Obviously, we have $S_{1,K}=S_K$.

The following property is not difficult to prove.
\begin{lemma}\label{75}(\cite{Sch14})
Let $K,K_i\in\mathcal{K}^n_o$ with $K_i\rightarrow K$ as $i\rightarrow\infty$ and $p\geq 1$. Then,
$$\Pi_pK_i\rightarrow \Pi_pK,\ \ \ K^\circ_i\rightarrow K^\circ,\ \ \ \mathrm{as}\ i\rightarrow\infty.$$
\end{lemma}

Let $K\in \mathcal{K}^n$. The polar formula for the volume of $K$ is
\begin{equation}\label{58}
  |K|=\frac{1}{n}\int_{S^{n-1}}\rho_K^n(\xi)d\xi=\frac{1}{n}\int_{S^{n-1}}|\xi|_{K}^{-n}d\xi.
\end{equation}

Say that $K$ is origin-symmetric if $K=-K$ and that $K$ is of class $C^2_+$ if its boundary $\partial K$ is of class $C^2$ and all principal curvatures of its boundary are positive and finite.

It is obvious that the $L^p$ projection body and the $L^p$ centroid body are both origin-symmetric. And, the $L^p$ centroid body is of class $C^2_+$ as demonstrated by Lutwak-Yang-Zhang:
\begin{lemma}\label{11}
(\cite{LYZ00}) If $K\in \mathcal{K}^n_o$, then the centroid body $\Gamma_p K$ is of $C^2_+$ class and is origin-symmetric.
\end{lemma}
We can compute the derivative of its support function:
\begin{equation}\label{103}
  \nabla'\Big(\frac{1}{p}h^p_{\tilde{\Gamma}_p K}\Big)(x)=\int_{S^{n-1}} \rho^{n+p}_K(\xi)|\xi\cdot x|^{p-1}\mathrm{sgn}(\xi\cdot x)\xi d\xi,\quad x\in\rn,
\end{equation}
where $\mathrm{sgn}(\cdot)$ denotes the sign function, $\tilde{\Gamma}_p K$ is the non-standard $L^p$ centroid body (see (\ref{72})) and $\nabla'$ denotes the gradient operator in $\rn$.
Throughout this paper, we denote the gradient operator in $\mathbb{R}^{n-1}$ by $\nabla$.

\subsection{The continuous Steiner symmetrization}
Given $K\in \mathcal{K}^n$, $\xi\in S^{n-1}$ and $y\in \xi^\bot$, define the section $K_y\subseteq \mathrm{span}\{\xi\}$ of $K$ with respect to $\xi$ as
\begin{equation*}
  K_y=\Big\{s\in \mathrm{span}\{\xi\}:(y,s)\in K\Big\}.
\end{equation*}
By Fubini's theorem, the volume of $K$ can be given by
\begin{equation}\label{102}
  |K|=\int_{\xi^\bot}|K_y|dy,
\end{equation}
where $|K_y|$ is the length of $K_y$.

Shadow systems, introduced by Rogers and Shephard \cite{RS58} and by Shephard \cite{Shephard}, are regarded as rearrangements of sections of the convex body roughly.
In this paper, the core tool is the continuous Steiner symmetrization (the parallel chord movement) of convex bodies as a special case of shadow systems in the following.
\begin{definition}\label{67}
(\cite{MSY,MY23,Sch14})Given some $\xi\in S^{n-1}$, a convex body $K$ and $\forall y\in K|{\xi^\bot}$, let $L_\xi^y$ be the line in the direction $\xi$ that passes through $y$ such that $L_\xi^y\cap K$ is a compact interval
$[c_y-l_y,c_y+l_y]$, called a chord. Let $I=[0,2]$. Define a continuous version of Steiner symmetrization $\{S_\xi^tK\}_{t\in I}$ by
$$S_\xi^tK\cap L_\xi^y=y+((1-t)c_y+[-l_y,l_y])\xi,\ \ \ l_y=\frac{1}{2}|L_\xi^y\cap K|\ \ \ \forall y\in K|{\xi^\bot}$$
(with $S_\xi^tK\cap L_\xi^y=\emptyset$ for $y\notin K|{\xi^\bot}$).
Here $|L_\xi^y\cap K|$ is the length of $L_\xi^y\cap K$, $\xi^\bot$ is the subspace orthogonal to $\xi$, and $ K|{\xi^\bot}$ is the orthogonal projection of $K$ on $\xi^\bot$.

\end{definition}
In particular, $S_\xi^0K$ is exactly equal to $K$, $S_\xi^2K$ is the reflection of $K$ with respect to the hyperplane $\xi^\bot$ and $S^1_\xi K=S_\xi K$ that is the classical Steiner symmetrization of $K$ with respect to the hyperplane $\xi^\bot$.

If we decompose $\rn=\xi^\bot\times\mathrm{span}\{\xi\},$ there are two concave functions $f,g:K|\xi^\bot\rightarrow\mathbb{R}$ such that
\begin{equation}\label{63}
  K=\Big\{(x,\lambda):-g(x)\leq\lambda\leq f(x),x\in K|\xi^\bot\Big\}.
\end{equation}
Therefore, the continuous Steiner symmetrization $\{S_\xi^tK\}_{t\in[0,2]}$ can be represented by
\begin{equation}\label{62}
  S_\xi^tK=\Big\{(x,\lambda):-\Big((1-\frac{t}{2})g(x)+\frac{t}{2}f(x)\Big)\leq\lambda\leq (1-\frac{t}{2})f(x)+\frac{t}{2}g(x),x\in K|\xi^\bot\Big\}
\end{equation}
and $(1-\frac{t}{2})f(x)+\frac{t}{2}g(x),(1-\frac{t}{2})g(x)+\frac{t}{2}f(x)$ are concave, so $S_\xi^tK$ is also a convex body and $|S_\xi^tK|=|K| $ by Fubini's theorem.

Obviously, the convex body $K$ is origin-symmetric if and only if $g(-x)=f(x)$.

Thus, for an origin-symmetric convex body $K$, we have
$$\frac{t}{2}f(-x)+(1-\frac{t}{2})g(-x)=(1-\frac{t}{2})f(x)+\frac{t}{2}g(x).$$
So the continuous Steiner symmetrization $\{S_\xi^t K\}_{t\in[0,2]}$ is origin-symmetric.

Let $\sigma_{\xi^\bot}(K)$ be the reflection of $K$ with respect to the hyperplane $\xi^\bot$,
we have the following lemma.
\begin{lemma}\label{68}
Let $t\in [0,2]$, $\lambda=2-t$ and $\xi\in S^{n-1}$. Then $S^\lambda_\xi K=\sigma_{\xi^\bot}(S^t_\xi K)$.
\end{lemma}
\begin{proof}
Let $t>1$, $\lambda=2-t$, $\xi\in S^{n-1}$ and $y\in K|\xi^\perp$. By Definition \ref{67}, one has
\begin{eqnarray*}
  S_\xi^\lambda K\cap L_\xi^y&=&y+((1-\lambda)c_y+[-l_y,l_y])\xi \\
   &=& y+(-(1-t)c_y+[-l_y,l_y])\xi \\
   &=& \sigma_{\xi^\bot}(S_\xi^tK\cap L_\xi^y),
\end{eqnarray*}
for $y\in K|{\xi^\bot}$. Thus it implies $S^\lambda_\xi K=\sigma_{\xi^\bot}(S^t_\xi K)$.
\end{proof}

There is the continuity property for the continuous Steiner symmetrization.
\begin{lemma}\label{74}
Let $t_0,t\in [0,2], \xi\in S^{n-1}$ and $K$ be a convex body. Then,
$$S_\xi^tK\rightarrow S_\xi^{t_0}K,\ \ \ \mathrm{as}\ t\rightarrow t_0.$$
\end{lemma}
\begin{proof}
Let any $\epsilon>0$ and $0\leq|t-t_0|\max_{y\in K|{\xi^\bot} }|c_y|=\delta<\epsilon$. By Definition \ref{67}, for each $x\in S_\xi^tK$ there is a $y\in K|{\xi^\bot}$ such that
\begin{eqnarray*}
 x\in S_\xi^tK\cap L_\xi^y&=&y+((1-t)c_y+[-l_y,l_y])\xi\\
   &=& y+[(1-t_0)c_y-l_y+(t_0-t)c_y,(1-t_0)c_y+l_y+(t_0-t)c_y]\xi \\
   &\subseteq& y+[(1-t_0)c_y-l_y,(1-t_0)c_y+l_y]\xi+\delta B^n\\
   &\subseteq&(S_\xi^{t_0}K\cap L_\xi^y)+\delta B^n\\
   &\subseteq& S_\xi^{t_0}K+\epsilon B^n,
\end{eqnarray*}
that is, $$S_\xi^tK\subseteq S_\xi^{t_0}K+\epsilon B^n.$$
Similarly, $$S_\xi^{t_0}K\subseteq S_\xi^{t}K+\epsilon B^n.$$
Therefore,
$$S_\xi^tK\rightarrow S_\xi^{t_0}K,\ \ \ \mathrm{as}\ t\rightarrow t_0.$$
\end{proof}

Finally, we recall the following property for the Steiner symmetrization of the $L^p$ polar projection body established by Lutwak-Yang-Zhang.
\begin{theorem}\label{59}
(\cite{LYZ00})Let $\xi\in S^{n-1}$, $p>1$ and $K\subseteq\rn$ be an origin-symmetric convex body of $C^2_+$ class. Then,
$$S_\xi{{\Pi}_p^\circ K}\subseteq {\Pi}_p^\circ S_\xi K,$$
with equality if and only if the chords of $K$ parallel to $\xi$ have midpoints that are coplanar.
\end{theorem}

\section{The $L^p$ Projection Centroid Conjecture}
\subsection{Some auxiliary lemmas}

Let $K\in\mathcal{K}^n_o$.
For simplicity, support functions  of the non-standard $L^p$ projection bodies $\tilde{\Pi}_p K$ and  the non-standard $L^p$ centroid bodies $\tilde{\Gamma}_p K$ of are defined by
$$h_{\tilde{\Pi}_p K}^p(x)=\int_{S^{n-1}} |\xi\cdot x|^pdS_{p,K}(\xi),\quad x\in\rn,$$
and
\begin{equation}\label{72}
  h_{\tilde{\Gamma}_p K}^p(x)=\int_{S^{n-1}} \rho^{n+p}_K(\xi)|\xi\cdot x|^pd\xi,\quad x\in\rn,
\end{equation}
respectively.

Note that $\tilde{\Pi}_p K$ is equivalent to ${\Pi}_p K$ up to a constant and $\tilde{\Gamma}_p K$ is equivalent to ${\Gamma}_p K$ up to a constant.

Firstly, we recall a useful notion defined by Lutwak-Yang-Zhang \cite{LYZ00}.

Suppose $A$ is the interior of a convex subset of $\xi^{\perp}=\mathbb{R}^{n-1}$ and $f:A\rightarrow\mathbb{R}$ is a $C^1$ function. Then, $\langle f\rangle:A\rightarrow \mathbb{R}$
is the function defined by
$$\langle f\rangle(x)=f(x)-\nabla f(x)\cdot x,\quad \forall x\in A.$$
Notice that $\langle\cdot\rangle$ is a linear operator; i.e., if $f_1,f_2:A\rightarrow\mathbb R$ and $\lambda_1,\lambda_2\in\mathbb R$, then
$$
\langle\lambda_1f_1+\lambda_2f_2\rangle=\lambda_1\langle f_1\rangle+\lambda_2\langle f_2\rangle.
$$
And we have
\begin{equation}\label{101}
  \langle f\rangle(x)=0,~\forall x\in A\Longrightarrow f\ \mathrm{is\ linear\ on}\ A.
\end{equation}
If $K$ is a convex body defined by (\ref{63}), then we can get
\begin{equation}\label{71}
  \langle f\rangle(x)=h_K(-\nabla f(x),1)\ \ \ \mathrm{and}\ \ \ \langle g\rangle(x)=h_K(-\nabla g(x),-1)
\end{equation}
for all $x\in\mathrm{int}(K|{\xi^\bot})$. Further, if $K$ is origin-symmetric, then we have
\begin{equation}\label{88}
  f(x)=g(-x),\ g(x)=f(-x),\ \langle f\rangle(x)=\langle g\rangle(-x),\ \langle g\rangle(x)=\langle f\rangle(-x),
\end{equation}
for all $x\in\mathrm{int}(K|{\xi^\bot})$.

The following lemmas are crucial in our paper.
\begin{lemma}\label{60}
(\cite{LYZ00})If $a,b\geq 0$ and $c,d>0$, then for $p>1$
$$(a+b)^p(c+d)^{1-p}\leq a^pc^{1-p}+b^pd^{1-p},$$
with equality if and only if $ad=bc$.
\end{lemma}

For the analogy of Lemma 9 in \cite{LYZ00}, one has the following.
\begin{lemma}\label{61}
Let $t\in[0,2]$, $\xi\in S^{n-1}$ and $K,L\subseteq\xi^\bot\times \mathrm{span}\{\xi\}$ be convex bodies of $\rn$. Then,
$$S^t_{\xi}K^\circ\subseteq L^\circ$$
if and only if
$$|(x,\lambda)|_{K^\circ}=|(x,-s)|_{K^\circ}=1,\ \mathrm{with}\ \lambda\ne-s\Longrightarrow |(x,(1-\frac{t}{2})\lambda+\frac{t}{2}s)|_{L^\circ}\leq 1~\&~|(x,(\frac{t}{2}-1)s-\frac{t}{2}\lambda)|_{L^\circ}\leq 1.$$
In addition, $S^t_{\xi}K^\circ=L^\circ$ if and only if
$$|(x,\lambda)|_{K^\circ}=|(x,-s)|_{K^\circ}=1,\ \mathrm{with}\ \lambda\ne-s\Longrightarrow |(x,(1-\frac{t}{2})\lambda+\frac{t}{2}s)|_{L^\circ}= 1~\&~|(x,(\frac{t}{2}-1)s-\frac{t}{2}\lambda)|_{L^\circ}=1.$$
\end{lemma}
\begin{proof}
Let us first prove the sufficiency direction.

Denote $K^\circ,S_\xi^tK^\circ$ by
$$K^\circ=\Big\{(x,\lambda):-g(x)\leq\lambda\leq f(x),x\in K|\xi^\bot\Big\}$$
and
$$S_\xi^tK^\circ=\Big\{(x,\lambda):-\Big((1-\frac{t}{2})g(x)+\frac{t}{2}f(x)\Big)\leq\lambda\leq (1-\frac{t}{2})f(x)+\frac{t}{2}g(x),x\in K|\xi^\bot\Big\}.$$
If $|(x,\lambda)|_{K^\circ}=|(x,-s)|_{K^\circ}=1$, then $\lambda=f(x),s=g(x)$ or $\lambda=-g(x),s=-f(x)$. Thus, $$|(x,(1-\frac{t}{2})f(x)+\frac{t}{2}g(x))|_{L^\circ}\leq 1~\&~|(x,(\frac{t}{2}-1)g(x)-\frac{t}{2}f(x))|_{L^\circ}\leq 1.$$
So  the boundary points $(x,(1-\frac{t}{2})f+\frac{t}{2}g),(x,(\frac{t}{2}-1)g-\frac{t}{2}f)$ of $S_\xi^tK^\circ$ is contained in $L^\circ$. Thus, by the convexity of $S_\xi^tK^\circ$, the chord from $(x,(1-\frac{t}{2})f+\frac{t}{2}g)$ to $(x,(\frac{t}{2}-1)g-\frac{t}{2}f)$ parallel to $\xi$ is contained in $L^\circ$. This implies
$$S^t_{\xi}K^\circ\subseteq L^\circ.$$

Let us now prove the necessity direction.

Let $|(x,\lambda)|_{K^\circ}=|(x,-s)|_{K^\circ}=1$. Since $y_1=(x,(\frac{t}{2}-1)s-\frac{t}{2}\lambda),y_2=(x,(1-\frac{t}{2})\lambda+\frac{t}{2}s)$ are boundary points of $S_\xi^tK^\circ$, we have
$$y_1,y_2\in S^t_{\xi}K^\circ\subseteq L^\circ.$$
Therefore, we obtain the desired result.

Now let us prove the equality condition.

 $S^t_{\xi}K^\circ=L^\circ$ if and only if the above $y_1,y_2\in\partial S^t_{\xi}K^\circ=\partial L^\circ$. This is equivalent to $$|y_1|_{L^\circ}=|y_2|_{L^\circ}=1.$$
And this is again equivalent to
$$
|(x,(1-\frac{t}{2})\lambda+\frac{t}{2}s)|_{L^\circ}= 1~\&~|(x,(\frac{t}{2}-1)s-\frac{t}{2}\lambda)|_{L^\circ}=1.
$$
\end{proof}
\begin{lemma}\label{79}
(\cite{LYZ00})Let $K$ be an origin-symmetric convex body of class $C^2_+$, given by
\begin{equation*}
  K=\Big\{(x,\lambda)\in\rn:-g(x)\leq\lambda\leq f(x),x\in K|\xi^\bot\Big\}.
\end{equation*}
Then for each $(y,\lambda)\in\xi^{\perp}\times\mathrm{span}\{\xi\}$,
\begin{equation}\label{64}
  h^p_{\tilde{\Pi}_pK}(y,\lambda)=2\int_{\mathrm{int}(K|\xi^\bot)}|\lambda-y\cdot\nabla f(x)|^p\langle f\rangle^{1-p}(x)dx,
\end{equation}
and
\begin{equation}\label{66}
h^p_{\tilde{\Pi}_pK}(y,\lambda)=2\int_{\mathrm{int}(K|\xi^\bot)}|\lambda+y\cdot\nabla g(x)|^p\langle g\rangle^{1-p}(x)dx.
\end{equation}
\end{lemma}

\begin{lemma}\label{Steinerformula}
Let $K$ be an origin-symmetric convex body of class $C^2_+$, given by
\begin{equation*}
  K=\Big\{(x,\lambda)\in\rn:-g(x)\leq\lambda\leq f(x),x\in K|\xi^\bot\Big\}.
\end{equation*}
Then
\begin{equation*}
  S_\xi^tK=\Big\{(x,\lambda):-\Big((1-\frac{t}{2})g(x)+\frac{t}{2}f(x)\Big)\leq\lambda\leq (1-\frac{t}{2})f(x)+\frac{t}{2}g(x),x\in K|\xi^\bot\Big\},
\end{equation*}
for $t\in[0,2]$.

Also, for each $(y,\lambda)\in\xi^{\perp}\times\mathrm{span}\{\xi\}$,
\begin{equation}\label{64S}
  h^p_{\tilde{\Pi}_pS_\xi^tK}(y,\lambda)=2\int_{\mathrm{int}(K|\xi^\bot)}\vline\lambda-y\cdot\nabla \left((1-\frac{t}{2})f(x)+\frac{t}{2}g(x)\right)\vline^p\langle (1-\frac{t}{2})f+\frac{t}{2}g\rangle^{1-p}(x)dx,
\end{equation}
and
\begin{equation}\label{66S}
h^p_{\tilde{\Pi}_pS_\xi^tK}(y,\lambda)=2\int_{\mathrm{int}(K|\xi^\bot)}\vline\lambda+y\cdot\nabla \left((1-\frac{t}{2})g(x)+\frac{t}{2}f(x)\right)\vline^p\langle (1-\frac{t}{2})g+\frac{t}{2}f\rangle^{1-p}(x)dx.
\end{equation}
\end{lemma}
\begin{proof}

Let $t\in[0,2]$ and $\xi\in S^{n-1}$.
Since $K$ is of $C^2_+$, there are $C^2$ concave functions $f,g:K|\xi^\bot\rightarrow\mathbb{R}$ with
\begin{equation*}
  K=\Big\{(x,\lambda):-g(x)\leq\lambda\leq f(x),x\in K|\xi^\bot\Big\},
\end{equation*}
and
\begin{equation*}
  S_\xi^tK=\Big\{(x,\lambda):-\Big((1-\frac{t}{2})g(x)+\frac{t}{2}f(x)\Big)\leq\lambda\leq (1-\frac{t}{2})f(x)+\frac{t}{2}g(x),x\in K|\xi^\bot\Big\},
\end{equation*}
such that the matrices $-\nabla^2 f,-\nabla^2g$ are positive definite on $\mathrm{int}(K|\xi^\bot)$

Then, the functions $f_1=(1-\frac{t}{2})f+\frac{t}{2}g$ and $g_1=(1-\frac{t}{2})g+\frac{t}{2}f$ are of $C^2$. And the matrices $-\nabla^2 f_1,-\nabla^2g_1$ are positive definite on $\mathrm{int}(K|\xi^\bot)$ for $t\in[0,2]$.

Thus, the graph surface of $f_1,g_1$ are of $C^2_+$ on $\mathrm{int}(K|\xi^\bot)$.

On the other hand, by Definition \ref{67}, since the $(n-2)$-dimensional compact surface $C=\big\{(x,f(x)):x\in \partial (K|\xi^\bot)\big\}=\big\{(x,c_x):x\in \partial (K|\xi^\bot)\big\}$ is of $C^2$, then one has $\mathcal{H}^{n-1}(C)=0$.

And note that
$C_t=\big\{(x,f_1(x)):x\in \partial (K|\xi^\bot)\big\}=\big\{(x,(1-t)c_x):x\in \partial (K|\xi^\bot)\big\}$ is the linear transformation of $C$, so we have $\mathcal{H}^{n-1}(C_t)=0$.

Since $S^t_\xi K$ is origin-symmetric, by \eqref{64} and \eqref{66}, we get that formulae (\ref{64S}) and (\ref{66S}) hold for $S^t_\xi K$.
\end{proof}

\subsection{The monotonicity of $|(\tilde{\Pi}_p^\circ S_\xi^t K)_y|$ with respect to $t$}
Firstly, we extend the inclusion relation in Theorem \ref{59} to the continuous Steiner symmetrization:

\begin{theorem}\label{65}
Let $t\in[0,2], \xi\in S^{n-1},p>1$ and $K$ be an origin-symmetric convex body of $C^2_+$. Then,
$$S_\xi^t{\tilde{\Pi}_p^\circ K}\subseteq \tilde{\Pi}_p^\circ S_\xi^t K.$$
\end{theorem}
\begin{proof}
The cases $t=0$ and $t=2$ are trivial. So we consider the remaining cases.

Suppose $|(y,\lambda)|_{\tilde{\Pi}_p^\circ K}=|(y,-s)|_{\tilde{\Pi}_p^\circ K}=1,\ \mathrm{with}\ \lambda\ne-s$ and $0<t<2$, and note that $h^p_{\tilde{\Pi}_p K}(y,\lambda)=h^p_{\tilde{\Pi}_p K }(y,-s)=1$ and $\langle f\rangle,\langle g\rangle>0$.
Using formula (\ref{64S}), the triangle inequality and Lemma \ref{60}, we compute that
{\small
\begin{eqnarray}
\nonumber   |(y,(1-\frac{t}{2})\lambda+\frac{t}{2}s)|_{\tilde{\Pi}_p^\circ S_\xi^t K}^p &=& h^p_{\tilde{\Pi}_p S_\xi^tK}(y,(1-\frac{t}{2})\lambda+\frac{t}{2}s)\\
\nonumber  &=&2\int_{\mathrm{int}(K|\xi^\bot)}|(1-\frac{t}{2})\lambda+\frac{t}{2}s-y\cdot\nabla \big((1-\frac{t}{2})f(x)+\frac{t}{2}g(x)\big)|^p\langle (1-\frac{t}{2})f+\frac{t}{2}g\rangle^{1-p}(x)dx\\
\nonumber   &\leq& 2\int_{\mathrm{int}(K|\xi^\bot)}\big(|(1-\frac{t}{2})(\lambda-y\cdot\nabla f(x))|+|\frac{t}{2}(s-y\cdot\nabla g(x))|\big)^p\big((1-\frac{t}{2})\langle f\rangle(x)+\frac{t}{2}\langle g\rangle(x)\big)^{1-p}dx \\
\nonumber   &\leq&  2\int_{\mathrm{int}(K|\xi^\bot)}|(1-\frac{t}{2})(\lambda-y\cdot\nabla f(x))|^p\big((1-\frac{t}{2})\langle f\rangle(x)\big)^{1-p}\\ \label{70}
\nonumber   &+& |\frac{t}{2}(s-y\cdot\nabla g(x))|^p\big(\frac{t}{2}\langle g\rangle(x)\big)^{1-p}dx\\
\nonumber   &=&2\int_{\mathrm{int}(K|\xi^\bot)}(1-\frac{t}{2})|\lambda-y\cdot\nabla f(x)|^p\langle f\rangle^{1-p}(x)+ \frac{t}{2}|s-y\cdot\nabla g(x)|^p\langle g\rangle^{1-p}(x)dx\\
\nonumber   &=&(1-\frac{t}{2})h^p_{\tilde{\Pi}_p K}(y,\lambda)+\frac{t}{2}h^p_{\tilde{\Pi}_p K}(y,-s)\\
\nonumber   &=&1.
\end{eqnarray}
}
 Similarly, it also follows from formula (\ref{66S}), the triangle inequality and Lemma \ref{60} that
 \begin{eqnarray*}
 \nonumber |(y,(\frac{t}{2}-1)s-\frac{t}{2}\lambda)|_{\tilde{\Pi}_p^\circ S_\xi^tK}^p &=& h^p_{\tilde{\Pi}_p S_\xi^tK}(y,(\frac{t}{2}-1)s-\frac{t}{2}\lambda)\\
 \nonumber &=&2\int_{\mathrm{int}(K|\xi^\bot)}|(1-\frac{t}{2})s+\frac{t}{2}\lambda+y\cdot\nabla \big((1-\frac{t}{2})f(x)+\frac{t}{2}g(x)\big)|^p\langle (1-\frac{t}{2})f+\frac{t}{2}g\rangle^{1-p}(x)dx\\
   &\leq&1.
\end{eqnarray*}
Thus, by the Lemma \ref{61}, one has
$$S_\xi^t{\tilde{\Pi}_p^\circ K}\subseteq \tilde{\Pi}_p^\circ S_\xi^t K.$$
\end{proof}

Now, consider sections of the $L^p$ polar projection body $\tilde{\Pi}_p^\circ S_\xi^t K$, denoted by
\begin{equation}\label{110}
  (\tilde{\Pi}_p^\circ S_\xi^t K)_z=\Big\{s\in \mathrm{span}\{\xi\}:(z,s)\in \tilde{\Pi}_p^\circ S_\xi^t K\Big\},
\end{equation}
for $\xi\in S^{n-1}$ and $z\in\xi^\perp.$ We have the following monotonicity result.
\begin{corollary}\label{69}
Let $K$ be an origin-symmetric convex body of $C^2_+$ and $\xi\in S^{n-1},p>1$. Given $z\in\xi^\bot$, the function $t\longmapsto |(\tilde{\Pi}_p^\circ S_\xi^t K)_z|,t\in[0,2]$, is non-decreasing on $[0,1]$ and non-increasing on $[1,2]$.
If $h(t-1)=|(\tilde{\Pi}_p^\circ S_\xi^t K)_z|$, then $h(s)$ is an even function on $[-1,1]$.
\end{corollary}

\begin{proof}
Choosing any $0<t_1<t_2\leq1$, set $\lambda=1-\frac{1-t_2}{1-t_1}$, and one has $0<\lambda\leq1$. From Definition \ref{67}, then
\begin{eqnarray*}
  (S^\lambda_\xi S^{t_1}_\xi K )\cap L_\xi^y &=& y+((1-\lambda)(1-t_1)c_y+[-l_y,l_y])\xi  \\
   &=&  y+((1-t_2)c_y+[-l_y,l_y])\xi\\
   &=&  S^{t_2}_\xi K \cap L_\xi^y,
\end{eqnarray*}
for $y\in K|{\xi^\bot}$. Thus, this implies $$S^\lambda_\xi S^{t_1}_\xi K =S^{t_2}_\xi K.$$
By  Lemma \ref{75},  Lemma \ref{74} and the density of the class of origin-symmetric convex bodies of $C^2_+$ in $\mathcal{K}^n_e$ (see e.g., \cite{Gardner06}), generally,
 Theorem \ref{65}  also holds on $\mathcal{K}^n_e$.
Together with $\lambda\in(0,1]$,  it follows that
$$S_\xi^{\lambda}{\tilde{\Pi}_p^\circ S^{t_1}_\xi K}\subseteq \tilde{\Pi}_p^\circ S_\xi^{\lambda} S_\xi^{t_1} K=\tilde{\Pi}_p^\circ S_\xi^{t_2} K.$$
Then, for $z\in \xi^\bot$
$$\big(S_\xi^{\lambda}{\tilde{\Pi}_p^\circ S^{t_1}_\xi K}\big)_z\subseteq \big(\tilde{\Pi}_p^\circ S_\xi^{t_2} K\big)_z.$$
That is
$$\big|\big({\tilde{\Pi}_p^\circ S^{t_1}_\xi K}\big)_z\big|\leq \big|\big(\tilde{\Pi}_p^\circ S_\xi^{t_2} K\big)_z\big|.$$
Using Theorem \ref{65} again, one has
$$\big|\big({\tilde{\Pi}_p^\circ  K}\big)_z\big|\leq \big|\big(\tilde{\Pi}_p^\circ S_\xi^{t_2} K\big)_z\big|.$$
Thus,  $t\longmapsto |(\tilde{\Pi}_p^\circ S_\xi^t K)_z|$ is non-decreasing on $[0,1]$.

Now, if $2\geq t_2>t_1\geq1$, set $\lambda_1=2-t_1,\lambda_2=2-t_2$, so $0\leq\lambda_2<\lambda_1\leq1$. By Lemma \ref{68} and the covariant property of the $L^p$ polar projection operator (see Lemma 2 in \cite{LYZ00} and (0.37) in \cite{Gardner06}), one has
$$|(\tilde{\Pi}^\circ_p S^{t_2}_\xi K)_z|=|(\tilde{\Pi}^\circ_p S^{\lambda_2}_\xi K)_z|\leq|(\tilde{\Pi}^\circ_p S^{\lambda_1}_\xi K)_z|=|(\tilde{\Pi}^\circ_p S^{t_1}_\xi K)_z|,$$
that is, $t\longmapsto |(\tilde{\Pi}_p^\circ S_\xi^t K)_z|$ is non-increasing on $[1,2]$.

Also by Lemma \ref{68}, $h(s)=|(\tilde{\Pi}_p^\circ S_\xi^{s+1} K)_z|$ is obviously an even function on $[-1,1]$.
\end{proof}

\subsection{The first order derivative of $|\tilde{\Pi}_p^\circ S_\xi^t K|$ with respect to $t$}

Now, we prove the key variation formula for $|\tilde{\Pi}_p^\circ S_\xi^t K|$ with respect to $t$ as follows:
\begin{proposition}\label{73}
Let $K$ be an origin-symmetric convex body of class $C^2_+$ denoted by (\ref{63}), $\xi\in S^{n-1}$ and $p>1$. Then,
{\small
$$\frac{d|\tilde{\Pi}^\circ_pS^t_\xi K|}{dt}=\int_{\mathrm{int}(K|\xi^\bot)} \nabla'\Big(h_{\tilde{\Gamma}_p\tilde{\Pi}^\circ_pS^t_\xi K}\Big)(\theta_t)\cdot(\nabla g-\nabla f,0)h^{1-p}_{S^t_\xi K}(\theta_t)h^{p-1}_{\tilde{\Gamma}_p\tilde{\Pi}^\circ_pS^t_\xi K}(\theta_t)+
     \frac{p-1}{p}\langle g-f\rangle h^{-p}_{S^t_\xi K}(\theta_t)h^p_{\tilde{\Gamma}_p\tilde{\Pi}^\circ_pS^t_\xi K}(\theta_t) dx$$
}
for $t\in[0,2]$, where $\theta_t=(-\big((1-\frac{t}{2})\nabla f+\frac{t}{2}\nabla g\big),1)$ (the derivative is assumed to be one-sided at $t=0,2$).
In particular, we have
{\small
$$\frac{d|\tilde{\Pi}^\circ_pS^t_\xi K|}{dt}\Big|_{t=0^+}=\int_{\mathrm{int}(K|\xi^\bot)} \nabla'\Big(h_{\tilde{\Gamma}_p\tilde{\Pi}^\circ_pK}\Big)(\theta_0)\cdot(\nabla g-\nabla f,0)h^{1-p}_{ K}(\theta_0)h^{p-1}_{\tilde{\Gamma}_p\tilde{\Pi}^\circ_p K}(\theta_0)+
     \frac{p-1}{p}\langle g-f\rangle h^{-p}_{ K}(\theta_0)h^p_{\tilde{\Gamma}_p\tilde{\Pi}^\circ_p K}(\theta_0) dx.$$
}

\end{proposition}
\begin{proof}
Let $(y,s)\in S^{n-1}\subseteq\xi^\bot\times \mathrm{span}\{\xi\}$. From formulae (\ref{58}) and (\ref{64S}), we have
\begin{eqnarray}
\nonumber  n|\tilde{\Pi}^\circ_pS^t_\xi K|&=&\int_{S^{n-1}}|(y,s)|^{-n}_{\tilde{\Pi}^\circ_pS^t_\xi K}d\mathcal{H}^{n-1} \\
\label{93}  &=& \int_{S^{n-1}}\big(h^p_{\tilde{\Pi}_pS^t_\xi K}(y,s)\big)^{-\frac{n}{p}}d\mathcal{H}^{n-1} \\
\nonumber   &=& \int_{S^{n-1}}\bigg(2\int_{\mathrm{int}(K|\xi^\bot)}|s-y\cdot\nabla \big((1-\frac{t}{2})f(x)+\frac{t}{2}g(x)\big)|^p\langle (1-\frac{t}{2})f+\frac{t}{2}g\rangle^{1-p}(x)dx\bigg)^{-\frac{n}{p}}d\mathcal{H}^{n-1}.
\end{eqnarray}

Set $\theta_t=(-\big((1-\frac{t}{2})\nabla f+\frac{t}{2}\nabla g\big),1)$, $\tilde{\theta}_t=\theta_t/|\theta_t|$ and $F_t=2|\theta_t\cdot(y,s)|^ph^{1-p}_{S^t_\xi K}(\theta_t)$.
Since $K$ is of $C^2_+$, then $\tilde{\theta}_t:\mathrm{int}(K|{\xi^\bot})\rightarrow S^{n-1}_+$ is a diffeomorphism of class $C^1$, where $ S^{n-1}_+\subseteq\xi^\bot\times \mathrm{span}\{\xi\}$ denotes the upper open hemisphere.

Fix $(y,s)\in S^{n-1}$ and $t_0\in[0,2]$, since $\{u\in S^{n-1}_+: u\cdot (y,s)=0\}$ is a $(n-2)$-dimensional submanifold, $\{x\in \mathrm{int}(K|{\xi^\bot}): \tilde{\theta}_{t_0}(x)\cdot (y,s)=0\}$ is a $(n-2)$-dimensional submanifold. This shows that
$$0=\mathcal{H}^{n-1}(\{x\in \mathrm{int}(K|{\xi^\bot}): \tilde{\theta}_{t_0}(x)\cdot (y,s)=0\})=\mathcal{H}^{n-1}(\{x\in \mathrm{int}(K|{\xi^\bot}): {\theta}_{t_0}(x)\cdot (y,s)=0\}).$$
That is,
\begin{equation}\label{116}
  \big|\{x\in \mathrm{int}(K|{\xi^\bot}): {\theta}_{t_0}(x)\cdot (y,s)=0\}\big|=0.
\end{equation}

Since $K$ is a convex body that contains the origin in its interior, from Lemma \ref{74} and Lemma \ref{75}, one has
$$h^\alpha_{S^t_\xi K}\rightarrow h^\alpha_{S^{t_0}_\xi K}$$
and
$$h^\alpha_{\tilde{\Pi}_pS^t_\xi K}\rightarrow h^\alpha_{\tilde{\Pi}_p S^{t_0}_\xi K}$$
uniformly on $S^{n-1}$,  as $t\rightarrow t_0$ and $\alpha\in\mathbb{R}\backslash\{0\}$.

So there is a sufficiently small constant $\delta>0$
 such that for $t\in(t_0-\delta,t_0+\delta)$~~ (or $t\in[0,\delta), t\in(2-\delta,2])$
\begin{equation}\label{104}
\begin{split}
  h^{1-p}_{S^t_\xi K}&\leq h^{1-p}_{S^{t_0}_\xi K}+1,\\
  h^{-p}_{S^t_\xi K}&\leq h^{-p}_{S^{t_0}_\xi K}+1,\\
  h^{-(n+p)}_{\tilde{\Pi}_pS^t_\xi K}&\leq h^{-(n+p)}_{\tilde{\Pi}_pS^{t_0}_\xi K}+1\ \ \mathrm{on}\ S^{n-1},
\end{split}
\end{equation}
and the sign functions
$$\mathrm{sgn}(\theta_t\cdot (y,s))=\mathrm{sgn}(\theta_{t_0}\cdot(y,s))$$
when $\theta_{t_0}\cdot(y,s)\ne 0$.

Thus, together with (\ref{116}), we have the function
$$t\longmapsto F_t$$
is continuously differentiable in $(t_0-\delta,t_0+\delta)$ for a.e. $x\in \mathrm{int}(K|{\xi^\bot})$.

And, by the mean value theorem, for each fixed $r\in(t_0-\delta,t_0+\delta)$ there is $\eta\in[t_0,r)$ (or $\eta\in(r,t_0]$) such that
$$F_r-F_{t_0}=(r-t_0)\frac{\partial F_t}{\partial t}\Big|_{t=\eta}.$$
This implies that
{\small
\begin{equation}\label{105}
 |F_r-F_{t_0}|=
|(r-t_0)\frac{\partial F_t}{\partial t}\Big|_{t=\eta}|,
\end{equation}
}
for a.e. $x\in \mathrm{int}(K|{\xi^\bot})$.

Now, by (\ref{71}) and (\ref{104}), for each $t\in(t_0-\delta,t_0+\delta)$ (or $t\in[0,\delta)$, $t\in(2-\delta,2]$) we get that
{\small
\begin{eqnarray*}
  |\frac{\partial F_t}{\partial t}| &=&   |p((\nabla f-\nabla g)\cdot y)\mathrm{sgn}(\theta_t\cdot (y,s)) |\theta_t\cdot(y,s)|^{p-1}h^{1-p}_{S^t_\xi K}(\theta_t)+(1-p)\langle g-f\rangle
  |\theta_t\cdot(y,s)|^{p}h^{-p}_{S^t_\xi K}(\theta_t) | \\
  &=&\big |p((\nabla f-\nabla g)\cdot y)\mathrm{sgn}(\theta_t\cdot (y,s)) \big|(\theta_t/|\theta_t|)\cdot(y,s)\big|^{p-1}h^{1-p}_{S^t_\xi K}(\theta_t/|\theta_t|)+(1-p)\langle g-f\rangle
  \big|(\theta_t/|\theta_t|)\cdot(y,s)\big|^{p}h^{-p}_{S^t_\xi K}(\theta_t/|\theta_t|)\big |\\
  &\leq& p|(\nabla f-\nabla g)\cdot y| \big(h^{1-p}_{S^{t_0}_\xi K}(\tilde{\theta}_{t})+1\big)+(p-1)|\langle g-f\rangle|\big(h^{-p}_{S^{t_0}_\xi K}(\tilde{\theta}_{t})+1\big) \\
  &\leq& p|(\nabla f-\nabla g)| \big((\min_{u\in S^{n-1}}h_{S^{t_0}_\xi K}(u))^{1-p}+1\big)+(p-1)|\langle g-f\rangle|\big((\min_{u\in S^{n-1}}h_{S^{t_0}_\xi K}(u))^{-p}+1\big)\\
  &\leq& c_2 (|\nabla f|+|\nabla g|)+c_3(\langle g\rangle+\langle f\rangle),
\end{eqnarray*}
}
for a.e. $x\in \mathrm{int}(K|{\xi^\bot})$, where $c_2,c_3$ are constants only depending on $t_0,p,n$ and $K$.

It is easy to see that $|\nabla f|,|\nabla g|,\langle g\rangle,\langle f\rangle\in L^1({\mathrm{int}(K|\xi^\bot)})$.
Thus, together with the dominated convergence theorem  and (\ref{105}), we have the derivative
$$\frac{\partial h^p_{\tilde{\Pi}_pS^t_\xi K}(y,s)}{\partial t}\Big|_{t=t_0}$$
exists for $t_0\in[0,2]$ (the one-sided derivative is assumed at $t_0=0,2$).
And for $t\in [t_0,r)$ (or $t\in(r,t_0]$), there is $\eta_t\in (t,r)$ (or $\eta_t\in(r,t)$) such that
\begin{eqnarray}
\nonumber  \bigg|\frac{\partial h^p_{\tilde{\Pi}_pS^t_\xi K}(y,s)}{\partial t}\bigg| &=& \Big|\lim_{r\rightarrow t^-} \int_{\mathrm{int}(K|\xi^\bot)} \frac{F_r-F_{t}}{r-t}dx\Big|\\
\nonumber   &\leq&  \int_{\mathrm{int}(K|\xi^\bot)} \lim_{r\rightarrow t^-}\big|\frac{\partial F_\zeta}{\partial \zeta}\Big|_{\zeta=\eta_t}\big|dx \\
\label{115}   &\leq&\int_{\mathrm{int}(K|\xi^\bot)} c_2 (|\nabla f|+|\nabla g|)+c_3(\langle g\rangle+\langle f\rangle)dx\\
\nonumber   &=& c_4.
\end{eqnarray}
Thus, using the mean value theorem, (\ref{104}) and (\ref{115}) (or replace (\ref{115}) by Lemma \ref{113}), there is $\eta'\in (t_0,r)$ (or $\eta'\in(r,t_0)$) such that
\begin{eqnarray*}
  |\big(h^p_{\tilde{\Pi}_pS^r_\xi K}(y,s)\big)^{-\frac{n}{p}}-\big(h^p_{\tilde{\Pi}_pS^{t_0}_\xi K}(y,s)\big)^{-\frac{n}{p}}| &=& \Big|\frac{\partial\big(h^p_{\tilde{\Pi}_pS^t_\xi K}(y,s)\big)^{-\frac{n}{p}}}{\partial t}\Big|_{t=\eta'}\Big||r-t_0| \\
   &=&  \frac{n}{p}h^{-(n+p)}_{\tilde{\Pi}_pS^{\eta'}_\xi K}(y,s) \Big|\frac{\partial h^p_{\tilde{\Pi}_pS^t_\xi K}(y,s)}{\partial t}\Big|_{t=\eta'}\Big||r-t_0|\\
   &\leq&\frac{nc_4}{p}\big(h^{-(n+p)}_{\tilde{\Pi}_pS^{t_0}_\xi K}(y,s)+1\big) |r-t_0| \\
   &\leq&\frac{nc_4}{p}\big((\min_{u\in S^{n-1}}h_{\tilde{\Pi}_pS^{t_0}_\xi K}(u))^{-(n+p)}+1\big) |r-t_0|\\
   &\leq& c_5|r-t_0|,
\end{eqnarray*}
for $t_0\in[0,2]$ where $c_5$ is a constant only depending on $t_0,p,n$ and $K$.

Therefore, by (\ref{93}), the dominated convergence theorem and Fubini's theorem, we get that
{\small
\begin{eqnarray}
\nonumber &\ & n\frac{d|\tilde{\Pi}^\circ_pS^t_\xi K|}{dt} = -\frac{n}{p}\int_{S^{n-1}}\big(h^p_{\tilde{\Pi}_pS^t_\xi K}(y,s)\big)^{-\frac{n}{p}-1}\frac{\partial h^p_{\tilde{\Pi}_pS^t_\xi K}(y,s)}{\partial t}d\mathcal{H}^{n-1}(y,s) \\
\nonumber   &=& -2\frac{n}{p}\int_{S^{n-1}}\rho_{\tilde{\Pi}^\circ_pS^t_\xi K}^{n+p}(y,s)\int_{\mathrm{int}(K|\xi^\bot)}\frac{\partial\big(|s-y\cdot\nabla \big((1-\frac{t}{2})f(x)+\frac{t}{2}g(x)\big)|^p\langle (1-\frac{t}{2})f+\frac{t}{2}g\rangle^{1-p}(x)\big)}{\partial t}dxd\mathcal{H}^{n-1}(y,s) \\
\nonumber   &=& -2\frac{n}{p}\int_{S^{n-1}}\rho_{\tilde{\Pi}^\circ_pS^t_\xi K}^{n+p}(y,s)\bigg(\int_{\mathrm{int}(K|\xi^\bot)}\frac{p}{2}|s-(1-\frac{t}{2})y\cdot \nabla f(x)-\frac{t}{2}y\cdot\nabla g(x)|^{p-1}\mathrm{sgn}\big(s-(1-\frac{t}{2})y\cdot \nabla f(x) \\
\nonumber   &-&\frac{t}{2}y\cdot\nabla g(x)\big)\big((1-\frac{t}{2})\langle f\rangle(x)+\frac{t}{2}\langle g\rangle(x)\big)^{1-p}(y\cdot\nabla f(x)-y\cdot\nabla g(x))dx+ \int_{\mathrm{int}(K|\xi^\bot)}\frac{1-p}{2}|s-(1-\frac{t}{2})y\cdot \nabla f(x)\\
\nonumber   &-&\frac{t}{2}y\cdot\nabla g(x)|^{p}\big((1-\frac{t}{2})\langle f\rangle(x)+\frac{t}{2}\langle g\rangle(x)\big)^{-p}(\langle g\rangle(x)-\langle f\rangle(x)) dx\bigg)d\mathcal{H}^{n-1}(y,s) \\
\nonumber   &=& -2\frac{n}{p}\int_{\mathrm{int}(K|\xi^\bot)}\bigg(\frac{p}{2}\big((1-\frac{t}{2})\langle f\rangle(x)+\frac{t}{2}\langle g\rangle(x)\big)^{1-p}(y\cdot\nabla f(x)-y\cdot\nabla g(x))\int_{S^{n-1}}\rho_{\tilde{\Pi}^\circ_pS^t_\xi K}^{n+p}(y,s)|s-(1-\frac{t}{2})y\cdot \nabla f(x)\\
\nonumber   &-&\frac{t}{2}y\cdot\nabla g(x)|^{p-1}\mathrm{sgn}\big(s-(1-\frac{t}{2})y\cdot \nabla f(x) -\frac{t}{2}y\cdot\nabla g(x)\big)d\mathcal{H}^{n-1}(y,s)+ \frac{1-p}{2}\big((1-\frac{t}{2})\langle f\rangle(x)\\
\nonumber   &+&\frac{t}{2}\langle g\rangle(x)\big)^{-p}(\langle g\rangle(x)-\langle f\rangle(x))\int_{S^{n-1}}\rho_{\tilde{\Pi}^\circ_pS^t_\xi K}^{n+p}(y,s)|s-(1-\frac{t}{2})y\cdot \nabla f(x)-\frac{t}{2}y\cdot\nabla g(x)|^{p} d\mathcal{H}^{n-1}(y,s)\bigg)dx.
\end{eqnarray}
}

Thus, together with (\ref{71}), (\ref{72}) and (\ref{103}), we get
{\footnotesize
\begin{eqnarray*}
\frac{d|\tilde{\Pi}^\circ_pS^t_\xi K|}{dt} &=& -\int_{\mathrm{int}(K|\xi^\bot)} (\nabla f-\nabla g,0)\cdot\bigg(\int_{S^{n-1}}\rho_{\tilde{\Pi}^\circ_pS^t_\xi K}^{n+p}(y,s)|(y,s)\cdot \theta_t|^{p-1}\mathrm{sgn}\big((y,s)\cdot \theta_t\big)(y,s)d\mathcal{H}^{n-1}(y,s)\bigg)\big(h_{S^t_\xi K}(\theta_t)\big)^{1-p}dx \\
&-&\frac{(1-p)}{p}\int_{\mathrm{int}(K|\xi^\bot)} \big(h_{S^t_\xi K}(\theta_t)\big)^{-p}(\langle g\rangle(x)-\langle f\rangle(x))\bigg(\int_{S^{n-1}}\rho_{\tilde{\Pi}^\circ_pS^t_\xi K}^{n+p}(y,s)|(y,s)\cdot \theta_t|^{p} d\mathcal{H}^{n-1}\bigg)dx \\
&=& \int_{\mathrm{int}(K|\xi^\bot)} (\nabla g-\nabla f,0)\cdot\nabla'\Big(\frac{1}{p}h^p_{\tilde{\Gamma}_p\tilde{\Pi}^\circ_pS^t_\xi K}\Big)(\theta_t)\big(h_{S^t_\xi K}(\theta_t)\big)^{1-p}dx +\frac{p-1}{p}\int_{\mathrm{int}(K|\xi^\bot)} \big(h_{S^t_\xi K}(\theta_t)\big)^{-p}(\langle g\rangle(x)
     -\langle f\rangle(x))h^p_{\tilde{\Gamma}_p\tilde{\Pi}^\circ_pS^t_\xi K}(\theta_t) dx \\
     &=&\int_{\mathrm{int}(K|\xi^\bot)} (\nabla g-\nabla f,0)\cdot\nabla'\Big(h_{\tilde{\Gamma}_p\tilde{\Pi}^\circ_pS^t_\xi K}\Big)(\theta_t)\big(h_{S^t_\xi K}(\theta_t)\big)^{1-p}h^{p-1}_{\tilde{\Gamma}_p\tilde{\Pi}^\circ_pS^t_\xi K}(\theta_t)dx+
     \frac{p-1}{p}\int_{\mathrm{int}(K|\xi^\bot)} \langle g-f\rangle h^{-p}_{S^t_\xi K}(\theta_t)h^p_{\tilde{\Gamma}_p\tilde{\Pi}^\circ_pS^t_\xi K}(\theta_t) dx.
\end{eqnarray*}
}
That is,
{\small
$$\frac{d|\tilde{\Pi}^\circ_pS^t_\xi K|}{dt}=\int_{\mathrm{int}(K|\xi^\bot)} \nabla'\Big(h_{\tilde{\Gamma}_p\tilde{\Pi}^\circ_pS^t_\xi K}\Big)(\theta_t)\cdot(\nabla g-\nabla f,0)h_{S^t_\xi K}^{1-p}(\theta_t)h^{p-1}_{\tilde{\Gamma}_p\tilde{\Pi}^\circ_pS^t_\xi K}(\theta_t)+
     \frac{p-1}{p}\langle g-f\rangle h^{-p}_{S^t_\xi K}(\theta_t)h^p_{\tilde{\Gamma}_p\tilde{\Pi}^\circ_pS^t_\xi K}(\theta_t) dx$$
     }
     for $t\in[0,2]$ (the one-sided derivative is assumed at $t=0,2$).
\end{proof}

\subsection{Continuous differentiability of $|(\tilde{\Pi}^\circ_pS^t_\xi K)_y|$}

Next, we prove the differentiability of the length of the section $(\tilde{\Pi}^\circ_pS^t_\xi K)_y$ for some $y\in \xi^\bot$ and each $t\in [0,2]$ (see (\ref{110})).
\begin{lemma}\label{113}
Let $K$ be an  origin-symmetric convex body of class $C^2_+$ and $\xi\in S^{n-1},p>1$. Then, the function $\rho_{\tilde{\Pi}^\circ_pS^t_\xi K}(v)$
is continuously differentiable for each $(t,v)\in[0,2]\times \rn\backslash\{0\}$ (the one-sided continuous differentiability is assumed at $t=0,2$). Thus, there is an open neighborhood $V(0)\subseteq\xi^\bot$ of $0\in\xi^\bot$ such that
$$|(\tilde{\Pi}^\circ_pS^t_\xi K)_y|$$
is continuously differentiable for each $(t,y)\in[0,2]\times V(0)$ (the one-sided continuous differentiability is assumed at $t=0,2$).
In particular, when $y=0$ we have
$$\frac{\partial |(\tilde{\Pi}^\circ_pS^t_\xi K)_0|}{\partial t}=2\frac{\partial\rho_{\tilde{\Pi}^\circ_pS^t_\xi K}(0,1)}{\partial t}.$$
\end{lemma}

\begin{proof}
Set $\theta_t=(-\big((1-\frac{t}{2})\nabla f+\frac{t}{2}\nabla g\big),1)$ and $\tilde{\theta}_t=\theta_t/|\theta_t|$.
From (\ref{64S})  and letting $(t,v)\rightarrow (t_0,v_0)\in[0,2]\times \rn$, we have formally
\begin{eqnarray*}
  \frac{\partial h^p_{\tilde{\Pi}_pS^t_\xi K}(v)}{\partial t} &=& \int_{\mathrm{int}(K|\xi^\bot)} p(\nabla f-\nabla g,0)\cdot v|v\cdot \theta_t|^{p-1}\mathrm{sgn}\big(v\cdot \theta_t\big)h_{S^t_\xi K}^{1-p}(\theta_t) +{(p-1)} h_{S^t_\xi K}^{-p}(\theta_t)(\langle f-g\rangle)|v\cdot \theta_t|^{p} dx \\
  &=&\int_{\mathrm{int}(K|\xi^\bot)} p(\nabla f-\nabla g,0)\cdot v|v\cdot \tilde{\theta}_t|^{p-1}\mathrm{sgn}\big(v\cdot \tilde{\theta}_t\big)h_{S^t_\xi K}^{1-p}(\tilde{\theta}_t) +{(p-1)} h_{S^t_\xi K}^{-p}(\tilde{\theta_t})(\langle f-g\rangle)|v\cdot \tilde{\theta}_t|^{p} dx
\end{eqnarray*}
(the one-sided derivative is assumed at $t=0,2$).

Since
\begin{eqnarray*}
   &\ & |p(\nabla f-\nabla g,0)\cdot v|v\cdot \tilde{\theta}_t|^{p-1}\mathrm{sgn}\big(v\cdot \tilde{\theta}_t\big)h_{S^t_\xi K}^{1-p}(\tilde{\theta}_t) +{(p-1)} h_{S^t_\xi K}^{-p}(\tilde{\theta_t})(\langle f-g\rangle)|v\cdot \tilde{\theta}_t|^{p}| \\
   &\leq&  p(|v_0|+1)^p(|\nabla f|+|\nabla g|) \big((\min_{u\in S^{n-1}}h_{S^{t_0}_\xi K}(u))^{1-p}+1\big)+(p-1)(|v_0|+1)^p\langle g+f\rangle\big((\min_{u\in S^{n-1}}h_{S^{t_0}_\xi K}(u))^{-p}+1\big).
\end{eqnarray*}
By the dominated convergence theorem, we have that $\frac{\partial h^p_{\tilde{\Pi}_pS^t_\xi K}(v)}{\partial t}$ is continuous for $(t,v)\in[0,2]\times \rn$.
Similarly, letting $(t,v)\rightarrow (t_0,v_0)$, one has formally
\begin{eqnarray*}
  \nabla'( h^p_{\tilde{\Pi}_pS^t_\xi K})(v) &=& 2p\int_{\mathrm{int}(K|\xi^\bot)} \theta_t|v\cdot \theta_t|^{p-1}\mathrm{sgn}\big(v\cdot \theta_t\big)h_{S^t_\xi K}^{1-p}(\theta_t)dx\\
  &=&2p\int_{\mathrm{int}(K|\xi^\bot)} {\theta}_t|v\cdot \tilde{\theta}_t|^{p-1}\mathrm{sgn}\big(v\cdot \theta_t\big)h_{S^t_\xi K}^{1-p}(\tilde{\theta}_t)dx.
\end{eqnarray*}
And then
\begin{eqnarray*}
   &\ & \big|{\theta}_t|v\cdot \tilde{\theta}_t|^{p-1}\mathrm{sgn}\big(v\cdot \theta_t\big)h_{S^t_\xi K}^{1-p}(\tilde{\theta}_t)\big| \\
   &\leq&  (|v_0|+1)^{p-1} \big((\min_{u\in S^{n-1}}h_{S^{t_0}_\xi}(u))^{1-p}+1\big)\sqrt{1+|\nabla \big((1-\frac{t}{2})f+\frac{t}{2}g\big)|^2}\\
   &\leq&  (|v_0|+1)^{p-1} \big((\min_{u\in S^{n-1}}h_{S^{t_0}_\xi}(u))^{1-p}+1\big)\Big(\sqrt{1+|\nabla \big((1-\frac{t_0}{2})f+\frac{t_0}{2}g\big)|^2}+1\Big).
\end{eqnarray*}
Because we have
$$\int_{\mathrm{int}(K|\xi^\bot)} \sqrt{1+|\nabla \big((1-\frac{t_0}{2})f+\frac{t_0}{2}g\big)|^2} dx=\frac{1}{2}P(S^{t_0}_\xi K)$$
where $P(S^{t_0}_\xi K)$ is the surface area of $S^{t_0}_\xi K$.

Thus by the dominated convergence theorem, $\nabla'( h^p_{\tilde{\Pi}_pS^t_\xi K})(v)$ is continuous for $(t,v)\in[0,2]\times \rn$.
Since  $$\rho_{\tilde{\Pi}^\circ_pS^t_\xi K}= \big(h^p_{\tilde{\Pi}_pS^t_\xi K}\big)^{-1/p},$$
this implies that $\rho_{\tilde{\Pi}^\circ_pS^t_\xi K}$ is continuously differentiable for each $(t,v)\in[0,2]\times \rn\backslash\{0\}$.

Now, consider the section $(\tilde{\Pi}^\circ_pS^t_\xi K)_y$ for each $y\in \xi^\bot$.
Since
$$\tilde{\Pi}^\circ_pS^t_\xi K\rightarrow \tilde{\Pi}^\circ_p S^{t_0}_\xi K$$ as $t\rightarrow t_0$, one has
$$\partial(\tilde{\Pi}^\circ_pS^t_\xi K)\rightarrow \partial(\tilde{\Pi}^\circ_p S^{t_0}_\xi K)$$ as $t\rightarrow t_0$.

Notice that $\tilde{\Pi}^\circ_pS^t_\xi K$ is origin-symmetric, so one of two boundary points of 0-section $\big(\partial(\tilde{\Pi}^\circ_pS^t_\xi K)\big)_0$ is positive and the other is negative.

Combined with   the finite covering property of $[0,2]$,
then there is  an open neighborhood $U(0)\subseteq\xi^\bot$ of $0\in\xi^\bot$
such that for each $t\in [0,2]$ and each $y\in U(0)$, there are two unique points $z,w\in\mathrm{int}(B^n|\xi^\bot)$ such that
\begin{equation}\label{109}
  y= \rho_{\tilde{\Pi}^\circ_pS^t_\xi K}{(z,\sqrt{1-|z|^2})}z=\rho_{\tilde{\Pi}^\circ_pS^t_\xi K}{(w,-\sqrt{1-|w|^2})}w
\end{equation}
and
\begin{equation}\label{111}
   |(\tilde{\Pi}^\circ_pS^t_\xi K)_y|=\rho_{\tilde{\Pi}^\circ_pS^t_\xi K}(z,\sqrt{1-|z|^2})\sqrt{1-|z|^2}+\rho_{\tilde{\Pi}^\circ_pS^t_\xi K}(w,-\sqrt{1-|w|^2})\sqrt{1-|w|^2}.
\end{equation}
By (\ref{109}), there exist functions
$$z=h(t,y)\ \ \ \mathrm{and}\ \ \ w=\tilde{h}(t,y).$$
Fix $t$ in (\ref{109}), one has
$$|z|\leq|y|(\min_{u\in S^{n-1}}\rho_{\tilde{\Pi}^\circ_p S^t_\xi K}(u))^{-1}\ \ \ \mathrm{and} \ \ \ |w|\leq|y|(\min_{u\in S^{n-1}}\rho_{\tilde{\Pi}^\circ_p S^t_\xi K}(u))^{-1}.$$
This shows that
\begin{equation}\label{112}
  h(t,0)=\tilde{h}(t,0)=0
\end{equation}
for each $t\in[0,2]$.

Since $ \partial(\tilde{\Pi}^\circ_pS^t_\xi K)\rightarrow \partial(\tilde{\Pi}^\circ_p  K)$ as $t\rightarrow 0^+$, this implies that
\begin{equation}\label{114}
  \lim_{t\rightarrow 0^+}h(t,y)=h(0,y)\ \ \ \mathrm{and}\ \ \ \lim_{t\rightarrow 0^+}\tilde{h}(t,y)=\tilde{h}(0,y)
\end{equation}
 for $y\in U(0)$.

 And, it is easy to see that
 \begin{equation}\label{117}
  \lim_{y\rightarrow y_0}h(0,y)=h(0,y_0)\ \ \ \mathrm{and}\ \ \ \lim_{y\rightarrow y_0}\tilde{h}(0,y)=\tilde{h}(0,y_0)
\end{equation}
for $y_0\in U(0)$.

And we note that similar results hold for the case $t=2$.

 Next, set
$$H(t,z,y)=\rho_{\tilde{\Pi}^\circ_pS^t_\xi K}{(z,\sqrt{1-|z|^2})}z-y$$
for each $(t,z,y)\in [0,2]\times \mathrm{int}(B^n|\xi^\bot)\times \xi^\bot$.
So it is easy to see that
$$\nabla_yH(t,z,y)=-I_{n-1},\ \ \ \frac{\partial H}{\partial t}(t,z,y)=\frac{\partial \rho_{\tilde{\Pi}^\circ_pS^t_\xi K}(z,\sqrt{1-|z|^2})}{\partial t}z$$
and
$$\nabla_zH(t,z,y)=z\otimes\Big(\nabla'(\rho_{\tilde{\Pi}^\circ_pS^t_\xi K})(z,\sqrt{1-|z|^2})\Big) \left(
                                                                                   \begin{array}{c}
                                                                                     I_{n-1} \\
                                                                                     -z(1-|z|^2)^{-\frac{1}{2}} \\
                                                                                   \end{array}
                                                                                 \right) +\rho_{\tilde{\Pi}^\circ_pS^t_\xi K}(z,\sqrt{1-|z|^2})I_{n-1}
$$
for each $(t,z,y)\in [0,2]\times \mathrm{int}(B^n|\xi^\bot)\times \xi^\bot$, where $I_{n-1}$ is the identity transformation on $\xi^\bot$ (the one-sided derivative is assumed at $t=0,2$).

By the continuous differentiability of $\rho_{\tilde{\Pi}^\circ_pS^t_\xi K}$, $H(t,z,y)$ is continuously differentiable.

Meanwhile, $H(t,0,0)=0$ and $\nabla_zH(t,0,0)=\rho_{\tilde{\Pi}^\circ_pS^t_\xi K}(0,1)I_{n-1}$ is nonsingular for each $t\in[0,2]$.

By the implicit function theorem, the finite covering property of $[0,2]$, (\ref{114}) and (\ref{117}), there is an open neighborhood $V_1(0)\subseteq U(0)$ such that
$z=h(t,y)$ is continuously differentiable for each $(t,y)\in [0,2]\times V_1(0)$.

Similarly, there is an open neighborhood $V_2(0)\subseteq U(0)$ such that $w=\tilde{h}(t,y)$ is continuously differentiable for each $(t,y)\in [0,2]\times V_2(0)$.

Thus, setting $V(0)=V_1(0)\cap V_2(0)$, it follows from (\ref{111}) and the chain-rule that
$$|(\tilde{\Pi}^\circ_pS^t_\xi K)_y|$$
is continuously differentiable for each $(t,y)\in[0,2]\times V(0)$ (the one-sided derivative is assumed at $t=0,2$).

Taking $y=0$, by (\ref{112}), (\ref{111}) and $\rho_{\tilde{\Pi}^\circ_pS^t_\xi K}(0,-1)=\rho_{\tilde{\Pi}^\circ_pS^t_\xi K}(0,1) $,

\begin{eqnarray*}
  \frac{\partial|(\tilde{\Pi}^\circ_pS^t_\xi K)_0|}{\partial t}&=&\frac{\partial\rho_{\tilde{\Pi}^\circ_pS^t_\xi K}(0,1)}{\partial t}+\frac{\partial\rho_{\tilde{\Pi}^\circ_pS^t_\xi K}(0,-1)}{\partial t} \\
   &=& 2 \frac{\partial\rho_{\tilde{\Pi}^\circ_pS^t_\xi K}(0,1)}{\partial t}.
\end{eqnarray*}
\end{proof}
\subsection{The convexity of $h^p_{\tilde{\Pi}_p S_\xi^{t} K}(y,s)$ with respect to $t$}
\begin{lemma}\label{91}
Let $K$ be an  origin-symmetric convex body of class $C^2_+$ and $\xi\in S^{n-1},p>1$.  For each fixed $(y,s)\in S^{n-1}\subseteq\xi^\bot\times\mathrm{span}\{\xi\}$,
the function $G(t)=h^p_{\tilde{\Pi}_p S_\xi^{t} K}(y,s),t\in[0,2]$ is convex and $h^p_{\tilde{\Pi}_p S_\xi^{2-t} K}(y,s)=h^p_{\tilde{\Pi}_p S_\xi^{t} K}(-y,s)$.
\end{lemma}

\begin{proof}
Let $t_1,t_2\in[0,2]$. From the representation (\ref{62}), we get
$$S_\xi^{\frac{t_1+t_2}{2}}K=\Big\{(x,\lambda):-\Big((1-\frac{t_1+t_2}{4})g(x)+\frac{t_1+t_2}{4}f(x)\Big)\leq\lambda\leq (1-\frac{t_1+t_2}{4})f(x)+\frac{t_1+t_2}{4}g(x),x\in K|\xi^\bot\Big\}.$$
Using formula (\ref{64S}), the triangle inequality and Lemma \ref{60}, we calculate that
\begin{eqnarray}
\nonumber&\ &h^p_{\tilde{\Pi}_p S_\xi^{(t_1+t_2)/2}K}(y,s) \\
\nonumber&=&2\int_{\mathrm{int}(K|\xi^\bot)}|s- (1-\frac{t_1+t_2}{4})y\cdot\nabla f-\frac{t_1+t_2}{4}y\cdot\nabla g|^p\langle (1-\frac{t_1+t_2}{4})f+\frac{t_1+t_2}{4}g\rangle^{1-p}dx\\
\nonumber   &\leq& 2\int_{\mathrm{int}(K|\xi^\bot)}\frac{1}{2}\big(|s- (1-\frac{t_1}{2})y\cdot\nabla f-\frac{t_1}{2}y\cdot\nabla g|+|s- (1-\frac{t_2}{2})y\cdot\nabla f-\frac{t_2}{2}y\cdot\nabla g)|\big)^p\big((1-\frac{t_1}{2})\langle f\rangle+\frac{t_1}{2}\langle g\rangle\\
\nonumber&+&(1-\frac{t_2}{2})\langle f\rangle+\frac{t_2}{2}\langle g\rangle\big)^{1-p}dx \\
\nonumber   &\leq&  \int_{\mathrm{int}(K|\xi^\bot)}|s- (1-\frac{t_1}{2})y\cdot\nabla f-\frac{t_1}{2}y\cdot\nabla g|^p\big((1-\frac{t_1}{2})\langle f\rangle+\frac{t_1}{2}\langle g\rangle\big)^{1-p}\\
\nonumber   &+& |s- (1-\frac{t_2}{2})y\cdot\nabla f-\frac{t_2}{2}y\cdot\nabla g|^p\big((1-\frac{t_2}{2})\langle f\rangle+\frac{t_2}{2}\langle g\rangle\big)^{1-p}dx\\
\nonumber   &=&\frac{1}{2}h^p_{\tilde{\Pi}_p S_\xi^{t_1}K}(y,s)+\frac{1}{2}h^p_{\tilde{\Pi}_p S_\xi^{t_2}K}(y,s). \\
 &\ & \label{108}
\end{eqnarray}
Together with the fact that the function $t\longmapsto h^p_{\tilde{\Pi}_p S_\xi^t K}(y,s)$ is continuous by Lemma \ref{113}, we get that
 $t\longmapsto h^p_{\tilde{\Pi}_p S_\xi^t K}(y,s)$ is convex by the classical Jensen's result \cite{Jensen}.

Next, let $t\in[0,2]$. It follows from (\ref{64S}), (\ref{88}), the change of variables $x=-z$ and $\mathrm{int}(K|\xi^\bot)$ is origin-symmetric that
\begin{eqnarray*}
 h^p_{\tilde{\Pi}_pS^{2-t}_\xi K}(y,s)&=& 2\int_{\mathrm{int}(K|\xi^\bot)}|s-y\cdot\nabla \big((1-\frac{2-t}{2})f(x)+\frac{2-t}{2}g(x)\big)|^p\langle (1-\frac{2-t}{2})f+\frac{2-t}{2}g\rangle^{1-p}(x)dx \\
 &=&2\int_{\mathrm{int}(K|\xi^\bot)}|s-y\cdot\nabla \big(\frac{t}{2}f(x)+(1-\frac{t}{2})g(x)\big)|^p \Big(\frac{t}{2}\langle f \rangle(x)+(1-\frac{t}{2})\langle g\rangle(x)\Big)^{1-p}dx\\
 &=&2\int_{\mathrm{int}(K|\xi^\bot)}|s-y\cdot\nabla \big(\frac{t}{2}g(-x)+(1-\frac{t}{2})f(-x)\big)|^p \Big(\frac{t}{2}\langle g \rangle(-x)+(1-\frac{t}{2})\langle f\rangle(-x)\Big)^{1-p}dx\\
 &=&2\int_{\mathrm{int}(K|\xi^\bot)}|s+y\cdot\nabla \big(\frac{t}{2}g(z)+(1-\frac{t}{2})f(z)\big)|^p \Big(\frac{t}{2}\langle g \rangle(z)+(1-\frac{t}{2})\langle f\rangle(z)\Big)^{1-p}dz\\
 &=&h^p_{\tilde{\Pi}_pS^{t}_\xi K}(-y,s).
\end{eqnarray*}

\end{proof}

\subsection{Proof of Theorem \ref{14}}

By homogeneity, the fact that there is a constant $c>0$ such that $$\Gamma_p \Pi_p^\circ K=cK,$$ is equivalent to the fact there exists a constant $c_1>0$ such that $$\tilde{\Gamma}_p \tilde{\Pi}_p^\circ K=c_1K.$$
And by Lemma \ref{11}, one has $K$ is of $C^2_+$ and origin-symmetric.
By this fact together with Proposition \ref{73} and $\nabla'h_K(\theta_0)=(x,f(x))$, we have
{\small

\begin{eqnarray*}
  \frac{d|\tilde{\Pi}^\circ_pS^t_\xi K|}{dt}\Big|_{t=0^+}&=&\int_{\mathrm{int}(K|\xi^\bot)} \nabla'\Big(h_{\tilde{\Gamma}_p\tilde{\Pi}^\circ_pK}\Big)(\theta_0)\cdot(\nabla g-\nabla f,0)h^{1-p}_{ K}(\theta_0)h^{p-1}_{\tilde{\Gamma}_p\tilde{\Pi}^\circ_p K}(\theta_0)+
     \frac{p-1}{p}\langle g-f\rangle h^{-p}_{ K}(\theta_0)h^p_{\tilde{\Gamma}_p\tilde{\Pi}^\circ_p K}(\theta_0) dx\\
   &=& \int_{\mathrm{int}(K|\xi^\bot)} \nabla'\Big(h_{c_1K}\Big)(\theta_0)\cdot(\nabla g-\nabla f,0)h^{1-p}_{ K}(\theta_0)h^{p-1}_{c_1 K}(\theta_0)+
     \frac{p-1}{p}\langle g-f\rangle h^{-p}_{ K}(\theta_0)h^p_{c_1 K}(\theta_0) dx\\
   &=& c_1^p\int_{\mathrm{int}(K|\xi^\bot)} \nabla'\Big(h_{K}\Big)(\theta_0)\cdot(\nabla g-\nabla f,0)+
     \frac{p-1}{p}\langle g-f\rangle  dx\\
   &=& c_1^p\int_{\mathrm{int}(K|\xi^\bot)} x\cdot(\nabla g(x)-\nabla f(x))+\frac{p-1}{p}(\langle g\rangle(x)-\langle f\rangle(x)) dx.
\end{eqnarray*}
}
Since $K$ is origin-symmetric, then $K|{\xi^\bot}$ is also origin-symmetric, and $x\cdot(\nabla g(x)-\nabla f(x)),\langle g\rangle(x)-\langle f\rangle(x)$ are both odd functions on $ \mathrm{int}(K|{\xi^\bot})$ by \eqref{88}.
Thus,
\begin{equation}\label{100}
  \frac{d|\tilde{\Pi}^\circ_pS^t_\xi K|}{dt}\Big|_{t=0^+}=0.
\end{equation}
Next using Corollary \ref{69} and Lemma \ref{113}, we have the function
$$t\longmapsto|(\tilde{\Pi}^\circ_pS^t_\xi K)_y|$$
is non-decreasing on $[0,1]$ and then there is  an open neighborhood $V(0)$ contained in $\xi^\bot$ such that for $y\in V(0)$ and $r\in [0,1)$, 
$$\frac{\partial |(\tilde{\Pi}^\circ_pS^t_\xi K)_y|}{\partial t}\Big|_{t=r^+}\geq 0.$$
Thus choosing $r=0$, it follows from (\ref{100}), (\ref{102}), Fatou's Lemma and the monotonicity that
\begin{eqnarray*}
  0=\frac{d|\tilde{\Pi}^\circ_pS^t_\xi K|}{dt}\Big|_{t=0^+}&=&\liminf_{t\rightarrow 0^+}\frac{|\tilde{\Pi}^\circ_pS^t_\xi K|-|\tilde{\Pi}^\circ_p K|}{t}\\
  &=&\liminf_{t\rightarrow 0^+}\int_{\xi^\bot}\frac{|(\tilde{\Pi}^\circ_pS^t_\xi K)_y|-|(\tilde{\Pi}^\circ_p K)_y|}{t}dy \\
  &\geq&\int_{\xi^\bot}\liminf_{t\rightarrow 0^+}\frac{|(\tilde{\Pi}^\circ_pS^t_\xi K)_y|-|(\tilde{\Pi}^\circ_p K)_y|}{t}dy \\
&\geq&\int_{V(0)}\liminf_{t\rightarrow 0^+}\frac{|(\tilde{\Pi}^\circ_pS^t_\xi K)_y|-|(\tilde{\Pi}^\circ_p K)_y|}{t}dy \\
   &=& \int_{V(0)}\frac{\partial|(\tilde{\Pi}^\circ_pS^t_\xi K)_y|}{\partial t}\Big|_{t=0^+}dy \\
   &\geq& 0.
\end{eqnarray*}
Thus, we have
$$\int_{V(0)}\frac{\partial|(\tilde{\Pi}^\circ_pS^t_\xi K)_y|}{\partial t}\Big|_{t=0^+}dy=0.$$
From also Lemma \ref{113}, then $\frac{\partial|(\tilde{\Pi}^\circ_pS^t_\xi K)_y|}{\partial t}\Big|_{t=0^+}$ is continuous at $y=0$, so one has
$$\frac{\partial|(\tilde{\Pi}^\circ_pS^t_\xi K)_0|}{\partial t}\Big|_{t=0^+}=0,$$
that is,
$$2\frac{\partial \rho_{\tilde{\Pi}^\circ_pS^t_\xi K}(0,1)}{\partial t}\Big|_{t=0^+}=0.$$
Then this implies
\begin{equation}\label{107}
  \frac{\partial \rho_{\tilde{\Pi}^\circ_pS^t_\xi K}(0,1)}{\partial t}\Big|_{t=0^+} =0.
\end{equation}
Again, from Corollary \ref{69}, Lemma \ref{113} and taking $y=0$, the differentiable function
$$t\longmapsto|(\tilde{\Pi}^\circ_pS^t_\xi K)_0|=2\rho_{\tilde{\Pi}^\circ_pS^t_\xi K}(0,1)$$
attains the maximum at $t=1$ on $[0,2]$. Then we have
$$\frac{\partial \rho_{\tilde{\Pi}^\circ_pS^t_\xi K}(0,1)}{\partial t}\Big|_{t=1} =0.$$
Combined with (\ref{107}) and the chain-rule, one has
$$\frac{\partial h^p_{\tilde{\Pi}_pS^t_\xi K}(0,1)}{\partial t}\Big|_{t=1}=\frac{\partial h^p_{\tilde{\Pi}_pS^t_\xi K}(0,1)}{\partial t}\Big|_{t=0^+}=0.$$
So, it follows from Lemma \ref{91} that for $r\in (0,1]$,
$$0=\frac{\partial h^p_{\tilde{\Pi}_pS^t_\xi K}(0,1)}{\partial t}\Big|_{t=0^+}\leq \frac{\partial h^p_{\tilde{\Pi}_pS^t_\xi K}(0,1)}{\partial t}\Big|_{t=r^+}\leq\frac{\partial h^p_{\tilde{\Pi}_pS^t_\xi K}(0,1)}{\partial t}\Big|_{t=1}=0 $$
and then
$$\frac{\partial h^p_{\tilde{\Pi}_pS^t_\xi K}(0,1)}{\partial t}\Big|_{t=r}=\frac{\partial h^p_{\tilde{\Pi}_pS^t_\xi K}(0,1)}{\partial t}\Big|_{t=r^+}=0.$$
This implies that
$$h^p_{\tilde{\Pi}_pS^0_\xi K}(0,1)=h^p_{\tilde{\Pi}_pS^1_\xi K}(0,1).$$
From Lemma \ref{91}, one has $h^p_{\tilde{\Pi}_pS^2_\xi K}(-y,s)=h^p_{\tilde{\Pi}_pS^0_\xi K}(y,s)$ for $(y,s)\in S^{n-1}$, then deducing
$$h^p_{\tilde{\Pi}_pS^2_\xi K}(0,1)=h^p_{\tilde{\Pi}_pS^0_\xi K}(0,1)=h^p_{\tilde{\Pi}_pS^1_\xi K}(0,1).$$
Thus taking $t_1=0$ and $t_2=2$ in (\ref{108}), there is the equality, so
$$\big(1+1\big)^p\big(\langle f\rangle
+\langle g\rangle\big)^{1-p}
  =  \langle f\rangle^{1-p}
   +\langle g\rangle^{1-p}.$$
Together with Lemma \ref{60}, we have
$$\langle g\rangle(x)=\langle f\rangle (x),$$
that is,
$$\langle g-f\rangle (x)=0$$
for each $x\in\mathrm{int}(K|\xi^\bot)$.

Combining with (\ref{101}), this shows that $g-f$ is linear on $\mathrm{int}(K|\xi^\bot)$ and hence that
the chords of $K$ parallel to $\xi$ have midpoints that are coplanar.

Finally, according to the arbitrariness of $\xi$, by the classical Bertrand-Brunn theorem (see e.g.,\cite{T96}), then $K$ must be an origin-symmetric ellipsoid.
{\hfill$\Box$}

\end{document}